\documentclass[11pt,a4paper,openany]{amsart}

\usepackage{amsmath,amsfonts,ams}
\usepackage{latexsym}
\usepackage{amssymb}
\usepackage{color}
\usepackage{epsfig}
\usepackage{graphicx}
\usepackage{subfigure}
\usepackage{mathrsfs}
\usepackage{amscd}
\usepackage{amsthm}
\usepackage{color}
\usepackage{cite}
\usepackage{epstopdf}

\numberwithin{equation}{section}
\newtheorem{proposition}{Proposition}[section]

\newtheorem{lemma}{Lemma}[section]
\newtheorem{theorem}{Theorem}[section]
\newtheorem{corollary}{Corollary}[section]

\begin{document}
\title[Maximal estimates for Schr\"odinger with inverse-square potential]{Maximal estimates for Schr\"odinger equation with inverse-square potential}
\author{Changxing Miao}
\address{Institute of Applied Physics and Computational Mathematics, P. O. Box 8009, Beijing, China, 100088}
\email{miao\_changxing@iapcm.ac.cn}

\author{Junyong Zhang}
\address{Department of Mathematics, Beijing Institute of Technology, Beijing 100081,
China} \email{zhang\_junyong@bit.edu.cn}

\author{Jiqiang Zheng}
\address{The Graduate School of China Academy of Engineering Physics, P. O. Box 2101, Beijing, China, 100088}
\email{zhengjiqiang@gmail.com}
\maketitle
\begin{abstract}
In this paper, we consider the maximal estimates for the solution to an
initial value problem of the linear Schr\"odinger equation with a singular
potential. We show a result about the pointwise convergence of
solutions to this special variable coefficient Schr\"odinger
equation with initial data $u_0\in H^{s}(\R^n)$ for $s>1/2$ or
radial initial data $u_0\in H^{s}(\R^n)$ for $s\geq1/4$ and the
solution does not converge when $s<1/4$.
\end{abstract}

\begin{center}
 \begin{minipage}{120mm}
 \end{minipage}
 \end{center}
\section{Introduction and Statement of Main Result}
We study the maximal estimates for the solution to an initial value
problem of the linear Schr\"odinger equation with an inverse square
potential. More precisely, we consider the following Schr\"odinger
equation
\begin{equation}\label{1.1}
\begin{cases}
i\partial_{t}u-\Delta u+\frac{a}{|x|^2}u=0, \qquad
(t,x)\in\R\times\R^n, ~a>-(n-2)^2/4,\\
u(x,0)=u_0(x).
\end{cases}
\end{equation}

The scale-covariance elliptic operator
$P_a:=-\Delta+\frac{a}{|x|^2}$ appearing in \eqref{1.1} plays a key
role  in many problems of physics and geometry. The heat and wave
flows for the elliptic operator $P_a$ have been studied in the
theory of combustion (see \cite{LZ}), and in  the wave propagation
on conic manifolds (see \cite{CT}). The Schr\"odinger equation
\eqref{1.1} arises in the study of quantum mechanics \cite{KSWW}.
There has been a lot of interest in developing Strichartz estimates
both for the Schr\"odinger and wave equations with the inverse
square potential, we refer the reader to Burq
etc.\cite{BPSS,BPSS1,PSS,PSS1} and the authors \cite{MZZ}. However,
as far as we known, there is few result about the maximal
estimates associated with the operator $P_a$, which arises in the
study of pointwise convergence problem for the Schr\"odinger and
wave equations with the inverse square potential. In this paper, we
aim to address some maximal estimates in the special settings
associated with the operator $P_a$. As a direct consequence, we
obtain the pointwise convergence result for $u_0\in H^{s}(\R^n)$
with $s>1/2$. \vspace{0.2cm}

In the case of the free Schr\"odinger equation without potential,
i.e. $a=0$, there are a large amount of literature in developing the
maximal estimate for its solution, which can be formally written as
\begin{equation*}
\begin{split}
u(t,x)=e^{it\Delta}u_0(x)=\int_{\R^n}e^{2\pi
i(x\cdot\xi-t|\xi|^2)}\hat{u}_0(\xi)\mathrm{d}\xi.
\end{split}
\end{equation*}
When $n=1$, Carleson \cite{Car} proved the convergence result
holds in sense of that $\lim\limits_{t\rightarrow 0}u(t)=u_0, a.e.
~x$ when $u_0\in H^s(\R)$ with $s\geq 1/4$. Dahlberg-Kenig \cite{DK}
showed that the result is sharp in the sense that the solution does
not converge when $s<1/4$. When $n\geq2$, Sj\"olin \cite{Sj} and
Vega \cite{Vega} independently proved the convergence results hold
when $u_0\in H^s(\R^n)$ when $s>1/2$. It follows from the construction of
Dahlberg-Kenig \cite{DK}, or alternatively Vega \cite {Vega} that the
solution does not converge when $s<1/4$. When $n=2$, Bourgain
\cite{Bo} showed that there is a certain $s<1/2$ such that the
convergence result holds, and this result was improved by Moyua-Vargas-Vega \cite{MVV}. Having shown the bilinear restriction
estimates for paraboloids, Tao-Vargas \cite{TV} and Tao \cite{Tao}
showed the convergence result holds for $s>15/32$ and $s>2/5$
respectively. The result was improved further to $s>3/8$ by Lee
\cite{Lee} and Shao \cite{Shao}. Very recently, Bourgain \cite{Bour}
made some progress in high dimension $n\geq2$ to show that the
convergence result holds for $s>1/2-1/(4n)$ when $n\geq1$ and the
convergence result needs $s\geq (n-2)/(2n)$ when
$n\geq5$.\vspace{0.2cm}

In the situation when $a\neq0$, the equation \eqref{1.1} can be
viewed as a special Schr\"odinger equation with variable singular
coefficients. The potential prevents us from using the Fourier
transform to give the expression of the solution. With the
motivation of  regarding the potential term as a perturbation on
angular direction in \cite{BPSS,PSS,MZZ}, we express the
solution by using the Hankel transform
of radial functions and spherical harmonics. Instead of
 Fourier transform, we utilize the Hankel transform and modify the argument of Vega \cite{Vega} to show the pointwise convergence result
holds when the initial data $u_0\in H^{s}(\R^n)$ for $s>1/2$, or
radial initial data $u_0\in H^{s}(\R^n)$ for $s\geq1/4$, and the
solution does not converge when $s<1/4$.
\vspace{0.2cm}

Let $u$ be the solution to \eqref{1.1}, we define the maximal
function by
\begin{equation}\label{1.2}
\begin{split} u^*(x)=\sup_{|t|>0}|u(x,t)|.
\end{split}
\end{equation}

Our main theorems are the following:
\begin{theorem}\label{thm1} Let $\beta>1$, $n\geq2$ and $s>\frac 12$.
Then
\begin{equation}\label{1.3}
\begin{split}
\int_{\R^n}| u^*(x)|^2\frac{\mathrm{d}x}{(1+|x|)^\beta}\leq
C\|u_0\|^2_{H^{s}(\R^n)}.
\end{split}
\end{equation}
\end{theorem}

As a direct consequence of Theorem \ref{thm1}, we have:
\begin{corollary}\label{pointwise}  Let $u_0\in H^s(\R^n)$ with $s>\frac12$ and $n\geq2$. Then
\begin{equation}\label{1.4}
\lim_{t\rightarrow 0}u(t,x)=u_0(x),\quad a.e.~~ x\in\R^n.
\end{equation}
\end{corollary}

\begin{theorem}\label{thm2}  Let $B^n$ be the open unit ball in $\R^n$. Assume that there exists a constant $C$
independent of $u_0$ such that
\begin{equation}\label{1.5}
\begin{split}
\int_{B^n}| u^*(x)|^2\mathrm{d}x\leq C\|u_0\|^2_{H^{s}(\R^n)}, \quad
\forall~ u_0(x)\in H^s(\R^n).
\end{split}
\end{equation}
Then $s\geq\frac14$.
\end{theorem}
With this in mind, Theorem \ref{thm1} is far from being sharp.
Assuming that the initial data possesses additional angular regularity,
we have
\begin{theorem}\label{thm3}  Let $B^n$ be the open unit ball in $\R^n$ and $\epsilon>0$. Then there exists a constant $C$
independent of $u_0$ such that
\begin{equation}
\begin{split}
\int_{B^n}| u^*(x)|^2\mathrm{d}x\leq
C\|u_0\|^2_{H^{\frac14}_{r}H^{\frac{n-1}2+\epsilon}_{\theta}},
\end{split}
\end{equation}
where for $s,s'\geq0$
\begin{equation*}
\begin{split}
H^{s}_{r}H^{s'}_{\theta}=\Big\{g:
\|g\|_{H^{s}_{r}H^{s'}_{\theta}}:=\big\|(1-\Delta_\theta)^{\frac
{s'}2}\big((1-\Delta)^{\frac
s2}g\big)\big\|_{L^{2}_{r^{n-1}\mathrm{d}r}(\R^+;L^{2}_{\theta}(\mathbb{S}^{n-1}))}\Big\}.
\end{split}
\end{equation*}
Here $\Delta_{\theta}$ denotes the Laplace-Beltrami operator on
$\mathbb{S}^{n-1}$.
\end{theorem}

{\bf Remarks:}

$\mathrm{i}).$ This result implies that the pointwise convergence of
solutions to \eqref{1.1} holds for radial initial data $u_0\in
H^{s}(\R^n)$ with $s\geq1/4$.\vspace{0.2cm}

$\mathrm{ii}).$ This result is an analogue of Theorem 1.1 in
\cite{CLS}. We remark that the parameter $\epsilon$ in \cite{CLS}
should be corrected for $\epsilon>1/2$ while not $\epsilon>0$. Thus,
we generalize and improve the result in \cite{CLS} by making use of
a finer result proved in \cite{GS}.\vspace{0.2cm}

Now we introduce some notations. We use $A\lesssim B$ to denote the
statement that $A\leq CB$ for some large constant $C$ which may vary
from line to line and depend on various parameters, and similarly
use $A\ll B$ to denote the statement $A\leq C^{-1} B$. We employ
$A\sim B$ to denote the statement that $A\lesssim B\lesssim A$. If
the constant $C$ depends on a special parameter other than the
above, we shall denote it explicitly by subscripts. We briefly write
$A+\epsilon$ as $A+$ or $A-\epsilon$ as $A-$ for $0<\epsilon\ll1$.
Throughout this paper, pairs of conjugate indices are written as $p,
p'$, where $\frac{1}p+\frac1{p'}=1$ with $1\leq
p\leq\infty$.\vspace{0.2cm}

This paper is organized as follows: In the section 2, we mainly
revisit the property of the Bessel functions and the Hankel
transforms associated with $-\Delta+\frac{a}{|x|^2}$. Section 3 is
devoted to the proofs of the theorems.\vspace{0.2cm}

{\bf Acknowledgments:}\quad  The authors thank the referee and the
associated editor for their invaluable comments and suggestions
which helped improve the paper greatly.  This work was supported in part by the NSF of China under grant No.11171033, No.11231006, and No.11371059.
The second author was partly
supported by the Fundamental Research Foundation of
BIT(20111742015) and RFDP(20121101120044).  C. Miao was also supported  by Beijing Center for Mathematics and Information
Interdisciplinary Sciences.

\section{Preliminary}

In this section, we first list some results about the Hankel
transform and the Bessel functions and then show a characterization
of Sobolev norm in the Hankel transform version. \vspace{0.2cm}

We begin with recalling the expansion formula with respect to the
spherical harmonics. For more details, we refer to Stein-Weiss
\cite{SW}. For the sake of convenience, let
\begin{equation}\label{2.1}
\xi=\rho \omega \quad\text{and}\quad x=r\theta\quad\text{with}\quad
\omega,\theta\in\mathbb{S}^{n-1}.
\end{equation}
For any $g\in L^2(\R^n)$, the expansion formula with respect to the
spherical harmonics yields
\begin{equation*}
g(x)=\sum_{k=0}^{\infty}\sum_{\ell=1}^{d(k)}a_{k,\ell}(r)Y_{k,\ell}(\theta)
\end{equation*}
where
\begin{equation*}
\{Y_{k,1},\ldots, Y_{k,d(k)}\}
\end{equation*}
is the orthogonal basis of the spherical harmonics space of degree
$k$ on $\mathbb{S}^{n-1}$, called $\mathcal{H}^{k}$, with the
dimension
\begin{equation*}
d(k)=\frac{2k+n-2}{k}C^{k-1}_{n+k-3}\simeq \langle k\rangle^{n-2}.
\end{equation*}
We remark that for $n=2$, the dimension of $\mathcal{H}^{k}$ is a
constant, which is independent of $k$. Obviously, we have the
orthogonal decomposition
\begin{equation*}
L^2(\mathbb{S}^{n-1})=\bigoplus_{k=0}^\infty \mathcal{H}^{k}.
\end{equation*}
By orthogonality, it gives
\begin{equation}\label{2.2}
\|g(x)\|_{L^2_\theta(\mathbb{S}^{n-1})}=\|a_{k,\ell}(r)\|_{\ell^2_{k,\ell}}.
\end{equation}
From
$-\Delta_{\theta}Y_{k,\ell}(\theta)=k(k+n-2)Y_{k,\ell}(\theta)$, the
fractional power of $1-\Delta_{\theta}$ can be written explicitly
\cite{MNNO}
\begin{equation}\label{a2.2}
(1-\Delta_{\theta})^{\frac
s2}g(x)=\sum_{k=0}^{\infty}\sum_{\ell=1}^{d(k)}(1+k(k+n-2))^{\frac
s2}a_{k,\ell}(r)Y_{k,\ell}(\theta).
\end{equation}

For our purpose, we need the Fourier transform of
$a_{k,\ell}(r)Y_{k,\ell}(\theta)$. Theorem 3.10 in \cite{SW} asserts
the Hankel transform formula
\begin{equation}\label{2.3}
\hat{g}(\rho\omega)\sim\sum_{k=0}^{\infty}\sum_{\ell=1}^{d(k)}
i^{k}Y_{k,\ell}(\omega)\rho^{-\frac{n-2}2}\int_0^\infty
J_{k+\frac{n-2}2}(2\pi r\rho)a_{k,\ell}(r)r^{\frac n2}\mathrm{d}r.
\end{equation}
Here the Bessel function $J_k(r)$ of order $k$ is defined by the
integral
\begin{equation*}
J_k(r)=\frac{(r/2)^k}{\Gamma(k+\frac12)\Gamma(1/2)}\int_{-1}^{1}e^{isr}(1-s^2)^{(2k-1)/2}\mathrm{d
}s\quad\text{with}~ k>-\frac12~\text{and}~ r>0.
\end{equation*}
A simple computation gives the rough estimates
\begin{equation}\label{2.4}
|J_k(r)|\leq
\frac{Cr^k}{2^k\Gamma(k+\frac12)\Gamma(1/2)}\left(1+\frac1{k+1/2}\right),
\end{equation}
where $C$ is a absolute constant. This estimate will be mainly used
when $r\lesssim1$. Another well known asymptotic expansion about the
Bessel function is
\begin{equation}\label{2.5}
J_k(r)=r^{-1/2}\sqrt{\frac2{\pi}}\cos(r-\frac{k\pi}2-\frac{\pi}4)+O_{k}(r^{-3/2}),\quad
\text{as}~ r\rightarrow\infty
\end{equation}
but with a constant depending on $k$ (see \cite{SW}). As pointed out
in \cite{Stein1}, if one seeks a uniform bound for large $r$ and
$k$, then the best one can do is $|J_k(r)|\leq C r^{-\frac13}$. One
will find that this decay doesn't lead to the desirable result.
Moreover, we recall the properties of Bessel function $J_k(r)$ in
\cite{ Stempak,Stein1}, we refer the readers to \cite{MZZ1} for the
detailed proof.
\begin{lemma}[Asymptotics of the Bessel function] \label{Bessel} Assume that $k\in \N$ and $k\gg1$. Let $J_k(r)$ be
the Bessel function of order $k$ defined as above. Then there exist
a large constant $C$ and small constant $c$ independent of $k$ and
$r$ such that:

$\bullet$ when $r\leq \frac k2$
\begin{equation}\label{2.6}
\begin{split}
|J_k(r)|\leq C e^{-c(k+r)};
\end{split}
\end{equation}

$\bullet$ when $\frac k 2\leq r\leq 2k$
\begin{equation}\label{2.7}
\begin{split}
|J_k(r)|\leq C k^{-\frac13}(k^{-\frac13}|r-k|+1)^{-\frac14};
\end{split}
\end{equation}

$\bullet$ when $r\geq 2k$
\begin{equation}\label{2.8}
\begin{split}
 J_k(r)=r^{-\frac12}\sum_{\pm}a_\pm(r,k) e^{\pm ir}+E(r,k),
\end{split}
\end{equation}
where $|a_\pm(r,k)|\leq C$ and $|E(r,k)|\leq Cr^{-1}$.
\end{lemma}

As a consequence of Lemma \ref{Bessel}, we have
\begin{lemma}\label{Bessel2} Let $R\gg1$. Then there exists a
constant $C$ independent of $k, R$ such that
\begin{equation}\label{2.9}
\int_R^{2R}|J_{k}(r)|^2\mathrm{d}r\leq C.
\end{equation}
\end{lemma}
\begin{proof}
To prove \eqref{2.9}, we write
\begin{equation*}
\begin{split}
\int_R^{2R}|J_{k}(r)|^2\mathrm{d}r=\int_{I_1}|J_{k}(r)|^2\mathrm{d}r
+\int_{I_2}|J_{k}(r)|^2\mathrm{d}r+\int_{I_3}|J_{k}(r)|^2\mathrm{d}r
\end{split}
\end{equation*}
where $I_1=[R,2R]\cap[0,\frac k 2], I_2=[R,2R]\cap[\frac k 2,2k]$
and $I_3=[R,2R]\cap[2k,\infty]$. By \eqref{2.6} and \eqref{2.8}, we
have
\begin{equation}\label{2.10}
\begin{split}
\int_{I_1}|J_{k}(r)|^2\mathrm{d}r\leq C
\int_{I_1}e^{-cr}\mathrm{d}r\leq C e^{-cR},
\end{split}
\end{equation}
and \begin{equation}\label{2.11}
\begin{split}\int_{I_3}|J_{k}(r)|^2\mathrm{d}r\leq C.\end{split}
\end{equation}
On the other hand, one has by \eqref{2.7}
\begin{equation*}
\begin{split}
\int_{[\frac k 2,2k]}|J_{k}(r)|^2\mathrm{d}r&\leq C \int_{[\frac k
2,2k]}k^{-\frac23}(1+k^{-\frac13}|r-k|)^{-\frac 12}\mathrm{d}r\leq
C.
\end{split}
\end{equation*}
Observing $[R,2R]\cap[\frac k 2,2k]=\emptyset$ unless $R\sim k$, we
obtain
\begin{equation}\label{2.12}
\begin{split}
\int_{I_2}|J_{k}(r)|^2\mathrm{d}r\leq C.
\end{split}
\end{equation}
This together with \eqref{2.10} and \eqref{2.11} yields \eqref{2.9}.
\end{proof}\vspace{0.2cm}

For simplicity, we define
\begin{equation}\label{2.13}
\mu(k)=\frac{n-2}2+k,\quad\text{and}\quad\nu(k)=\sqrt{\mu^2(k)+a}\quad\text{with}\quad
a>-(n-2)^2/4.
\end{equation}
We sometime briefly write $\nu$ as $\nu(k)$. Let $f$ be Schwartz
function defined on $\R^n$, we define the Hankel transform of order
$\nu$
\begin{equation}\label{2.14}
(\mathcal{H}_{\nu}f)(\xi)=\int_0^\infty(r\rho)^{-\frac{n-2}2}J_{\nu}(r\rho)f(r\omega)r^{n-1}\mathrm{d}r,
\end{equation}
where $\rho=|\xi|$, $\omega=\xi/|\xi|$ and $J_{\nu}$ is the Bessel
function of order $\nu$. In particular, the function $f$ is radial,
then we have
\begin{equation}\label{2.15}
(\mathcal{H}_{\nu}f)(\rho)=\int_0^\infty(r\rho)^{-\frac{n-2}2}J_{\nu}(r\rho)f(r)r^{n-1}\mathrm{d}r.
\end{equation}
If $
f(x)=\sum\limits_{k=0}^{\infty}\sum\limits_{\ell=1}^{d(k)}a_{k,\ell}(r)Y_{k,\ell}(\theta)
$, it follows from \eqref{2.3} that
\begin{equation}\label{2.16}
\begin{split}
\hat f(\xi)=\sum_{k=0}^{\infty}\sum_{\ell=1}^{d(k)}2\pi
i^{k}Y_{k,\ell}(\omega)\big(\mathcal{H}_{\mu(k)}a_{k,\ell}\big)(\rho).
\end{split}
\end{equation}

The following properties of the Hankel transform are obtained in
\cite{BPSS,PSS}:
\begin{lemma}\label{Hankel}
Let $\mathcal{H}_{\nu}$ be defined above and $
A_{\nu(k)}:=-\partial_r^2-\frac{n-1}r\partial_r+\big[\nu^2(k)-\big(\frac{n-2}2\big)^2\big]{r^{-2}}.
$ Then

$(\rm{i})$ $\mathcal{H}_{\nu}=\mathcal{H}_{\nu}^{-1}$,

$(\rm{ii})$ $\mathcal{H}_{\nu}$ is self-adjoint, i.e.
$\mathcal{H}_{\nu}=\mathcal{H}_{\nu}^*$,

$(\rm{iii})$ $\mathcal{H}_{\nu}$ is an $L^2$ isometry, i.e.
$\|\mathcal{H}_{\nu}\phi\|_{L^2_\xi}=\|\phi\|_{L^2_x}$,

$(\rm{iv})$ $\mathcal{H}_{\nu}(
A_{\nu}\phi)(\xi)=|\xi|^2(\mathcal{H}_{\nu} \phi)(\xi)$, for
$\phi\in L^2$.
\end{lemma}

We next recall the following almost orthogonality inequality. Denote
by $P_j$ and $\tilde{P}_j$ the usual dyadic frequency localization
at $|\xi|\sim 2^{j}$ and the localization with respect to
$\big(-\Delta+\frac a{|x|^2}\big)^{\frac12}$. We define the
projectors $M_{jj'}=P_j\tilde{P}_{j'}$ and
$N_{jj'}=\tilde{P}_jP_{j'}$.  More precisely, let $f$ be in  the
$k$'th harmonic subspace, then
\begin{equation*}
P_j f=\mathcal{H}_{\mu(k)}\beta_j
\mathcal{H}_{\mu(k)}f\quad\text{and}\quad \tilde{P}_j
f=\mathcal{H}_{\nu(k)}\beta_j \mathcal{H}_{\nu(k)}f,
\end{equation*}
where $\beta_j(\xi)=\beta(2^{-j}|\xi|)$ with $\beta\in
C_0^\infty(\R^+)$ supported in $[\frac12,2]$. Then we have the
following almost orthogonality inequality \cite{BPSS}:

\begin{lemma}[Almost orthogonality inequality]\label{orthogonality} Let $f\in L^2(\R^n)$, then there exists a constant $C$
independent of $j,j'$ such that
\begin{equation}\label{2.17}
\|M_{j j'}f\|_{L^2(\R^n)},~~ \|N_{j j'}f\|_{L^2(\R^n)}\leq C
2^{-\epsilon|j-j'|}\|f\|_{L^2(\R^n)},
\end{equation}
where $\epsilon<1+\min\{\frac{n-2}2,
(\frac{(n-2)^2}4+a)^{\frac12}\}$.
\end{lemma}
As a consequence, we have
\begin{lemma}\label{sobolev}
Let $f\in L^2(\R^n)$ such that $
f(x)=\sum\limits_{k=0}^{\infty}\sum\limits_{\ell=1}^{d(k)}a_{k,\ell}(r)Y_{k,\ell}(\theta)$.
Then for $0\leq s<1+\min\{\frac{n-2}2,
(\frac{(n-2)^2}4+a)^{\frac12}\}$ and $s'\geq0$
\begin{equation}\label{2.18}
\sum_{k=0}^{\infty}\sum_{\ell=1}^{d(k)}\sum_{M\in
2^{\Z}}M^{2s}(1+k)^{2s'}\|b_{k,\ell}(\rho)\chi(\frac{\rho}M)\rho^{\frac{n-1}{2}}\|_{L_\rho^2}^2\sim\|f\|_{\dot{H}^s_r{H}^{s'}_{\theta}}^2,
\end{equation}
where $b_{k,\ell}(\rho)=(\mathcal{H}_{\nu(k)}a_{k,\ell})(\rho)$ and
$\chi\in C_0^\infty(\R^n)$ such that $\text{supp}~\chi\subset
[1/2,1]$.
\end{lemma}
\begin{proof}
Note that $-\Delta_\theta Y_{k,\ell}=k(k+n-2)Y_{k,\ell}$, then we have by
Lemma \ref{Hankel}
\begin{equation*}
\begin{split}
\|f\|_{\dot{H}^0_r{H}^{s'}_{\theta}}^2&\sim\sum_{k=0}^\infty\sum_{\ell=1}^{d(k)}
(1+k)^{2s'}\|a_{k,\ell}(r)\|_{L^{2}_{r^{n-1}\mathrm{d}r}(\R^+)}^2\|Y_{k,\ell}(\theta)\|_{L_\theta^2}^2\\&\sim\sum_{k=0}^\infty\sum_{\ell=1}^{d(k)}
(1+k)^{2s'}\|b_{k,\ell}(\rho)\|_{L^{2}_{\rho^{n-1}\mathrm{d}\rho}(\R^+)}^2.
\end{split}
\end{equation*}
By \eqref{a2.2}, it suffices to show \eqref{2.18} with $s'=0$. By
Lemma \ref{Hankel}, we have
\begin{align*}
\|b_{k,\ell}(\rho)\chi(\frac{\rho}M)\rho^{\frac{n-1}{2}}\|_{L_\rho^2}
=&\big\|\chi(\frac{\rho}M)\mathcal{H}_{\nu}\big[Y_{k,l}(\theta)a_{k,\ell}(r)\big](\xi)\big\|_{L_\xi^2}\\
=&\Big\|\mathcal{H}_{\nu}\Big[\chi(\frac{\rho}M)\mathcal{H}_{\nu}\big(Y_{k,l}(\theta)a_{k,\ell}(r)\big)(\xi)\Big]\Big\|_{L_x^2}.
\end{align*}
This yields that by letting let $j=\log_2 M$
\begin{align*}
\big\|b_{k,\ell}(\rho)\chi(\frac{\rho}M)\rho^{\frac{n-1}{2}}\big\|_{L_\rho^2}
=&\big\|\big[\mathcal{H}_{\nu}\chi(\frac{\rho}M)\mathcal{H}_{\nu}\big]\big(Y_{k,l}(\theta)a_{k,\ell}(r)\big)\big\|_{L_x^2}\\
=&\big\|\tilde{P}_j\big(Y_{k,l}(\theta)a_{k,\ell}(r)\big)\big\|_{L_x^2}.
\end{align*}
Let $g_{k,\ell}(x)=Y_{k,l}(\theta)a_{k,\ell}(r)$ and
$\overline{P_{j'}}=P_{j'-1}+P_{j'}+P_{j'+1}$. We have by
the triangle inequality and Lemma \ref{orthogonality}
\begin{align*}
\text{L.H.S
of}~\eqref{2.18}&=\sum_{k=0}^{\infty}\sum_{\ell=1}^{d(k)}\sum_{j\in
\Z}2^{2sj}\big\|\tilde{P}_jg_{k,\ell}\big\|^2_{L_x^2}\\&\lesssim\sum_{k=0}^{\infty}\sum_{\ell=1}^{d(k)}\sum_{j\in
\Z}2^{2sj}\big(\sum_{j'}\big\|\tilde{P}_j\overline{P_{j'}}P_{j'}g_{k,\ell}\big\|_{L_x^2}\big)^{2}\\
&\lesssim\sum_{k=0}^{\infty}\sum_{\ell=1}^{d(k)}\sum_{j\in
\Z}2^{2sj}\big(\sum_{j'}2^{-\epsilon|j-j'|}\big\|P_{j'}g_{k,\ell}\big\|_{L_x^2}\big)^{2},
\end{align*}
where $s<\epsilon<1+\min\{\frac{n-2}2,
(\frac{(n-2)^2}4+a)^{\frac12}\}$. Let $0<\epsilon_1\ll1$ such that
$\epsilon_2:=\epsilon-\epsilon_1>s$, then
\begin{equation*}
\begin{split}
\text{L.H.S of}~\eqref{2.18}&\leq
C\sum_{k=0}^{\infty}\sum_{\ell=1}^{d(k)}\sum_{j\in\Z}2^{2js}\sum_{j'}2^{-2\epsilon_2|j-j'|}\|P_{j'}g_{k,\ell}\|^2_{L^2(\R^n)}\sum_{j'}2^{-2\epsilon_1|j-j'|}
\\&\leq
C\sum_{k=0}^{\infty}\sum_{\ell=1}^{d(k)}\sum_{j'}2^{2j's}\sum_{j\in\Z}2^{2js}2^{-2\epsilon_2|j|}\|P_{j'}g_{k,\ell}\|^2_{L^2(\R^n)}\\&\leq
C\sum_{k=0}^{\infty}\sum_{\ell=1}^{d(k)}\sum_{j'}2^{2j's}\|P_{j'}g_{k,\ell}\|^2_{L^2(\R^n)}.
\end{split}
\end{equation*}
By the definition of $P_{j'}$, Lemma \ref{Hankel} and \eqref{2.16},
we have
\begin{equation*}
\begin{split}
\text{L.H.S of}~\eqref{2.18}&\leq
C\sum_{j'}2^{2j's}\sum_{k=0}^{\infty}\sum_{\ell=1}^{d(k)}\big\|\chi(\frac{\rho}{2^{j'}})
\big[\mathcal{H}_{\mu(k)}a_{k,\ell}\big](\rho)\rho^{\frac{n-1}{2}}\big\|^2_{L^2(\R^+)}\\&=
C\sum_{j'}2^{2j's}\big\|\chi(\frac{\rho}{2^{j'}})\sum_{k=0}^{\infty}\sum_{\ell=1}^{d(k)}2\pi
i^{k}
\big[\mathcal{H}_{\mu(k)}a_{k,\ell}\big](\rho)Y_{k,\ell}(\omega)\big\|^2_{L^2(\R^n)}
\\&=
C\sum_{j'}2^{2j's}\big\|\chi(\frac{\rho}{2^{j'}})\hat
f\big\|^2_{L^2(\R^n)}\sim \|f\|^2_{\dot H^s}.
\end{split}
\end{equation*}
We can use the similar argument to prove
\begin{equation*}
\begin{split}
\text{L.H.S of}~\eqref{2.18}&\geq c \|f\|^2_{\dot H^s}.
\end{split}
\end{equation*} Therefore we conclude the proof of Lemma \ref{orthogonality}.
\end{proof}

\section{Proof of the Main Theorems}
In this section, we first use the spherical harmonic expansion to
write the solution as a linear combination of products of the Hankel
transform of radial functions and spherical harmonics. We prove the
main theorems by analyzing the property of the Hankel transform. The
key ingredients are to use the stationary phase argument and to exploit
the asymptotics behavior of the Bessel function.
\subsection{The expression of the solution.}
Consider the following Cauchy problem:
\begin{equation}\label{3.1}
\begin{cases}
i\partial_{t}u-\Delta u+\frac{a}{|x|^2}u=0,\\
u(x,0)=u_0(x).
\end{cases}
\end{equation}
We use the spherical harmonic expansion to write
\begin{equation}\label{3.2}
u_0(x)=\sum_{k=0}^{\infty}\sum_{\ell=1}^{d(k)}a^0_{k,\ell}(r)Y_{k,\ell}(\theta).
\end{equation}
Let us consider the equation \eqref{3.1} in polar coordinates. Write
$v(t,r,\theta)=u(t,r\theta)$ and $g(r,\theta)=u_0(r\theta)$. Then
$v(t,r,\theta)$ satisfies that
\begin{equation}\label{3.3}
\begin{cases}
i\partial_{t}
v-\partial_{rr}v-\frac{n-1}r \partial_r v-\frac1{r^2}\Delta_{\theta}v+\frac{a}{r^2}v=0\\
v(0,r,\theta)=g(r,\theta).
\end{cases}
\end{equation}
By \eqref{3.2}, we have
\begin{equation*}
g(r,\theta)=\sum_{k=0}^{\infty}\sum_{\ell=1}^{d(k)}a^0_{k,\ell}(r)Y_{k,\ell}(\theta).
\end{equation*}
Using separation of variables, we can write $v$ as a linear
combination of products of radial functions and spherical harmonics
\begin{equation}\label{3.4}
v(t,r,\theta)=\sum_{k=0}^{\infty}\sum_{\ell=1}^{d(k)}v_{k,\ell}(t,r)Y_{k,\ell}(\theta),
\end{equation}
where $v_{k,\ell}$ is given by
\begin{equation*}
\begin{cases}
i\partial_{t}
v_{k,\ell}-\partial_{rr}v_{k,\ell}-\frac{n-1}r\partial_rv_{k,\ell}+\frac{k(k+n-2)+a}{r^2}v_{k,\ell}=0, \\
v_{k,\ell}(0,r)=a^0_{k,\ell}(r)
\end{cases}
\end{equation*}
for each $k,\ell\in \N,~1\leq\ell\leq d(k)$. Then  it reduces to consider by the definition of $A_{\nu(k)}$
\begin{equation}\label{3.5}
\begin{cases}
i\partial_{t}
v_{k,\ell}+A_{\nu(k)}v_{k,\ell}=0, \\
v_{k,\ell}(0,r)=a^0_{k,\ell}(r).
\end{cases}
\end{equation}
Applying the Hankel transform to the equation \eqref{3.5}, by
$(\rm{iv})$ in Lemma \ref{Hankel}, we have
\begin{equation}\label{3.6}
\begin{cases}
i\partial_{t}
\tilde{ v}_{k,\ell}+\rho^2\tilde{v}_{k,\ell}=0 \\
\tilde{v}_{k,\ell}(0,\rho)=b^0_{k,\ell}(\rho),
\end{cases}
\end{equation}
where
\begin{equation}\label{3.7}
\tilde{v}_{k,\ell}(t,\rho)=(\mathcal{H}_{\nu}
v_{k,\ell})(t,\rho),\quad
b^0_{k,\ell}(\rho)=(\mathcal{H}_{\nu}a^0_{k,\ell})(\rho).
\end{equation}
Solving this ODE and inverting the Hankel transform, we obtain
\begin{equation*}
\begin{split}
v_{k,\ell}(t,r)&=\int_0^\infty(r\rho)^{-\frac{n-2}2}J_{\nu(k)}(r\rho)\tilde{v}_{k,\ell}(t,\rho)\rho^{n-1}\mathrm{d}\rho\\
&=\int_0^\infty(r\rho)^{-\frac{n-2}2}J_{\nu(k)}(r\rho)e^{
it\rho^2}b^0_{k,\ell}(\rho)\rho^{n-1}\mathrm{d}\rho.
\end{split}
\end{equation*}
Therefore we get
\begin{equation}\label{3.8}
\begin{split} &u(x,t)=v(t,r,\theta)
\\&=\sum_{k=0}^{\infty}\sum_{\ell=1}^{d(k)}Y_{k,\ell}(\theta)\int_0^\infty(r\rho)^{-\frac{n-2}2}J_{\nu(k)}(r\rho)e^{
it\rho^2}b^0_{k,\ell}(\rho)\rho^{n-1}\mathrm{d}\rho
\\&=\sum_{k=0}^{\infty}\sum_{\ell=1}^{d(k)}Y_{k,\ell}(\theta)\mathcal{H}_{\nu(k)}\big[e^{
it\rho^2}b^0_{k,\ell}(\rho)\big](r).
\end{split}
\end{equation}\vspace{0.2cm}

\subsection{Proof of the Theorem \ref{thm1}.}

In this subsection, we prove the Theorem \ref{thm1}. By the Sobolev
embedding $\dot H^{\frac12-}(\R)\cap \dot
H^{\frac12+}(\R)\hookrightarrow L^{\infty}(\R)$, it suffices to show
\begin{proposition}\label{pro1} Let $\alpha\geq \frac12-\frac\beta4$ and $\beta=1+$ such that $$2\alpha-1+\frac\beta2<1+\min\big\{{(n-2)}/2,
({(n-2)^2}/4+a)^{\frac12}\big\},$$ then there exists a constant $C$
independent of $u_0$ such that
\begin{equation}\label{3.9}
\begin{split}
\int_{\R^n}\int_{\R}|\partial_t^\alpha
u(x,t)|^2\frac{\mathrm{d}t\mathrm{d}x}{(1+|x|)^\beta}\leq
C\|u_0\|^2_{\dot H^{2\alpha-1+\frac\beta2}(\R^n)}.
\end{split}
\end{equation}
\end{proposition}
\begin{proof}By the Plancherel theorem with respect to time $t$, we obtain
\begin{equation*}
\begin{split}
&\int_{\R^n}\int_{\R}|\partial_t^\alpha
u(x,t)|^2\frac{\mathrm{d}t\mathrm{d}x}{(1+|x|)^\beta}=\int_{\R^n}\int_{\R}\big|\tau^\alpha\int_{\R}
e^{-it\tau}
u(x,t)\mathrm{d}t\big|^2\frac{\mathrm{d}\tau\mathrm{d}x}{(1+|x|)^\beta}.
\end{split}
\end{equation*}
Using \eqref{3.8}, we further have
\begin{equation*}
\begin{split}
&\text{L.H.S of}~\eqref{3.9}\\ \lesssim &
\int_{\R^{n+1}}\big|\tau^\alpha\sum_{k=0}^{\infty}\sum_{\ell=1}^{d(k)}Y_{k,\ell}(\theta)\int_{\R}
\int_0^\infty(r\rho)^{-\frac{n-2}2}J_{\nu(k)}(r\rho)e^{
it(\rho^2-\tau)}b^0_{k,\ell}(\rho)\rho^{n-1}\mathrm{d}\rho\mathrm{d}t\big|^2\frac{\mathrm{d}\tau\mathrm{d}x}{(1+|x|)^\beta}\\
\lesssim & \int_{\R^{n+1}}\big|\tau^\alpha
\sum_{k=0}^{\infty}\sum_{\ell=1}^{d(k)}Y_{k,\ell}(\theta)\int_0^\infty(r\rho)^{-\frac{n-2}2}J_{\nu(k)}(r\rho)b^0_{k,\ell}(\rho)
\rho^{n-1}\delta(\tau-\rho^2)\mathrm{d}\rho\big|^2\frac{\mathrm{d}\tau\mathrm{d}x}{(1+|x|)^\beta}\\
\lesssim&\int_{\R^n}\int_{0}^\infty\big|\sum_{k=0}^{\infty}\sum_{\ell=1}^{d(k)}Y_{k,\ell}(\theta)\rho^\alpha
(r\sqrt{\rho})^{-\frac{n-2}2}J_{\nu(k)}(r\sqrt{\rho})b^0_{k,\ell}(\sqrt{\rho})\rho^{\frac{n-1}2}\rho^{-\frac12}\big|^2\frac{\mathrm{d}\rho\mathrm{d}x}{(1+|x|)^\beta}.
\end{split}
\end{equation*}
By the orthogonality, we see that
\begin{equation}
\begin{split}
&\text{L.H.S
of}~\eqref{3.9}\\&\lesssim\sum_{k=0}^{\infty}\sum_{\ell=1}^{d(k)}\int_{0}^\infty\int_{0}^\infty\big|\rho^{2\alpha+\frac12}
(r\rho)^{-\frac{n-2}2}J_{\nu(k)}(r\rho)b^0_{k,\ell}(\rho)\rho^{n-2}\big|^2\frac{
\mathrm{d}\rho ~r^{n-1}\mathrm{d}r}{(1+r)^\beta}.
\end{split}
\end{equation}
Let $\chi$ be a smoothing function supported in $[1,2]$. For our
purpose, we make a dyadic decomposition to obtain
\begin{equation*}
\begin{split}
&\text{L.H.S of}~\eqref{3.9}\\
\lesssim&\sum_{k=0}^{\infty}\sum_{\ell=1}^{d(k)}\sum_{M\in2^{\Z}}\int_{0}^\infty\int_{0}^\infty\big|\rho^{2\alpha+\frac12}
(r\rho)^{-\frac{n-2}2}J_{\nu(k)}(r\rho)b^0_{k,\ell}(\rho)\rho^{n-2}\chi(\frac\rho
M)\big|^2\frac{ r^{n-1}\mathrm{d}r\mathrm{d}\rho}{(1+r)^\beta}
\\ \lesssim&\sum_{k=0}^{\infty}\sum_{\ell=1}^{d(k)}\sum_{M\in2^{\Z}}M^{2(n-2+2\alpha+\frac12)+1-n}\int_{0}^\infty\int_{0}^\infty\big|
(r\rho)^{-\frac{n-2}2}J_{\nu(k)}(r\rho)b^0_{k,\ell}(M\rho)\chi(\rho)\big|^2\frac{
r^{n-1}\mathrm{d}r\mathrm{d}\rho}{(1+\frac r M)^\beta}\\
\lesssim&\sum_{k=0}^{\infty}\sum_{\ell=1}^{d(k)}\sum_{M\in2^{\Z}}\sum_{R\in2^{\Z}}M^{n-2+4\alpha}R^{n-1}\int_{R}^{2R}\int_{0}^\infty\big|
(r\rho)^{-\frac{n-2}2}J_{\nu(k)}(r\rho)b^0_{k,\ell}(M\rho)\chi(\rho)\big|^2\frac{
\mathrm{d}r\mathrm{d}\rho}{(1+\frac r M)^\beta}.
\end{split}
\end{equation*}
Define
\begin{equation}\label{3.11}
\begin{split}
G_{k,\ell}(R,M)=\int_{R}^{2R}\int_{0}^\infty\big|
(r\rho)^{-\frac{n-2}2}J_{\nu(k)}(r\rho)b^0_{k,\ell}(M\rho)\chi(\rho)\big|^2\frac{
\mathrm{d}r\mathrm{d}\rho}{(1+\frac r M)^\beta}.
\end{split}
\end{equation}
\begin{proposition}\label{pro2} We have the following inequality
\begin{equation}\label{3.12}
G_{k,\ell}(R,M) \lesssim
\begin{cases}
R^{2\nu(k)-n+3}M^{-n}\min\Big\{1,\Big(\frac{M}{R}\Big)^\beta\Big\}\|b^0_{k,\ell}(\rho)\chi(\frac{\rho}M)\rho^{\frac{n-1}2}\|^2_{L^2},~
R\lesssim 1;\\
\min\Big\{1,\Big(\frac{M}{R}\Big)^\beta\Big\}R^{-(n-2)}M^{-n}\|b^0_{k,\ell}(\rho)\chi(\frac{\rho}M)\rho^{\frac{n-1}2}\|^2_{L^2},~
R\gg1.
\end{cases}
\end{equation}
\end{proposition}
\begin{proof} To prove \eqref{3.12}, we break it into two
cases.\vspace{0.2cm}

$\bullet$ Case 1: $R\lesssim1$. Since $\rho\sim1$, we have
$r\rho\lesssim1$. By the property of the Bessel function
\eqref{2.4}, we obtain
\begin{equation}\label{3.13}
\begin{split}
G_{k,\ell}(R,M)&\lesssim\int_{R}^{2R}\int_{0}^\infty\Big| \frac{
(r\rho)^{\nu(k)}(r\rho)^{-\frac{n-2}2}}{2^{\nu(k)}\Gamma(\nu(k)+\frac12)\Gamma(\frac12)}b^0_{k,\ell}(M\rho)\chi(\rho)\Big|^2\mathrm{d}\rho\frac{
\mathrm{d}r}{(1+\frac r M)^\beta}\\& \lesssim
R^{2\nu(k)-n+3}M^{-n}\min\Big\{1,\Big(\frac{M}{R}\Big)^\beta\Big\}\|b^0_{k,\ell}(\rho)\chi(\frac{\rho}M)\rho^{\frac{n-1}2}\|^2_{L^2}.
\end{split}
\end{equation}

$\bullet$ Case 2: $R\gg1$. Since $\rho\sim1$, we have $r\rho\gg 1$. We
estimate
\begin{equation}\label{3.14}
\begin{split}
G_{k,\ell}(R,M)&\lesssim
R^{-(n-2)}\int_{0}^\infty\big|b^0_{k,\ell}(M\rho)\chi(\rho)\big|^2\int_{R}^{2R}\big|J_{\nu(k)}(
r\rho)\big|^2 \frac{ \mathrm{d}r}{(1+\frac r
M)^\beta}\mathrm{d}\rho.
\end{split}
\end{equation}
$(i)$ Subcase: $R\lesssim M$. Noting that $\rho\sim1$, we obtain by
Lemma \ref{Bessel2}
\begin{equation}\label{3.15}
\begin{split}
\int_{R}^{2R}\big|J_{\nu(k)}( r\rho)\big|^2 \frac{
\mathrm{d}r}{(1+\frac r M)^\beta}\lesssim
\int_{R}^{2R}\big|J_{\nu(k)}( r\rho)\big|^2 \mathrm{d}r\lesssim 1.
\end{split}
\end{equation}
$(ii)$ Subcase: $R\gg M$. Noticing that $\rho\sim1$ again, we obtain
by Lemma \ref{Bessel2}
\begin{equation}\label{3.16}
\begin{split}
\int_{R}^{2R}\big|J_{\nu(k)}( r\rho)\big|^2 \frac{
\mathrm{d}r}{(1+\frac r M)^\beta}\lesssim
\Big(\frac{M}{R}\Big)^\beta \int_{R}^{2R}\big|J_{\nu(k)}(
r\rho)\big|^2 \mathrm{d}r\lesssim \Big(\frac{M}{R}\Big)^\beta.
\end{split}
\end{equation}
Putting \eqref{3.15} and \eqref{3.16} into \eqref{3.14}, we have
\begin{equation*}
\begin{split}
G_{k,\ell}(R,M)&\lesssim
\min\Big\{1,\Big(\frac{M}{R}\Big)^\beta\Big\}R^{-(n-2)}\int_{0}^\infty\big|b^0_{k,\ell}(M\rho)\chi(\rho)\big|^2\mathrm{d}\rho\\&
\lesssim
\min\Big\{1,\Big(\frac{M}{R}\Big)^\beta\Big\}R^{-(n-2)}M^{-n}\|b^0_{k,\ell}(\rho)\chi(\frac{\rho}M)\rho^{\frac{n-1}2}\|^2_{L^2}.
\end{split}
\end{equation*}
Thus we prove \eqref{3.12}.
\end{proof}
Now we return to prove Proposition \ref{pro1}. By Proposition \ref{pro2},
we show
\begin{equation*}
\begin{split}
&\int_{\R^n}\int_{\R}|\partial_t^\alpha
u(x,t)|^2\frac{\mathrm{d}t\mathrm{d}x}{(1+|x|)^\beta}\\&\lesssim\sum_{k=0}^{\infty}\sum_{\ell=1}^{d(k)}\sum_{M\in2^{\Z}}\sum_{\{R\in2^{\Z}:
R\lesssim1\}}M^{4\alpha-2}R^{2(\nu(k)+1)}\min\Big\{1,\Big(\frac{M}{R}\Big)^\beta\Big\}\|b^0_{k,\ell}(\rho)\chi(\frac{\rho}M)\rho^{\frac{n-1}2}\|^2_{L^2}
\\&+\sum_{k=0}^{\infty}\sum_{\ell=1}^{d(k)}\sum_{M\in2^{\Z}}\sum_{\{R\in2^{\Z}:
R\gg1\}}M^{4\alpha-2+\beta}
R^{1-\beta}\|b^0_{k,\ell}(\rho)\chi(\frac{\rho}M)\rho^{\frac{n-1}2}\|^2_{L^2}.
\end{split}
\end{equation*}
From $\beta=1+$, one has
\begin{equation*}
\begin{split}
&\sum_{M\in2^{\Z}}\sum_{\{R\in2^{\Z}:
R\lesssim1\}}M^{4\alpha-2}R^{2(\nu(k)+1)}\min\Big\{1,\Big(\frac{M}{R}\Big)^\beta\Big\}\|b^0_{k,\ell}(\rho)\chi(\frac{\rho}M)\rho^{\frac{n-1}2}\|^2_{L^2}\\&
\lesssim\sum_{M\in2^{\Z}}M^{4\alpha-2+\beta}\|b^0_{k,\ell}(\rho)\chi(\frac{\rho}M)\rho^{\frac{n-1}2}\|^2_{L^2}.
\end{split}
\end{equation*}
Since $\alpha\geq \frac12-\frac\beta4$, we have by Lemma
\ref{sobolev}
\begin{equation*}
\begin{split}
\int_{\R^n}\int_{\R}|\partial_t^\alpha
u(x,t)|^2\frac{\mathrm{d}t\mathrm{d}x}{(1+|x|)^\beta}&\lesssim\sum_{k=0}^{\infty}\sum_{\ell=1}^{d(k)}\sum_{M\in2^{\Z}}M^{4\alpha-2+\beta}
\|b^0_{k,\ell}(\rho)\chi(\frac{\rho}M)\rho^{\frac{n-1}2}\|^2_{L^2}\\&\leq
C\|u_0\|^2_{\dot H^{2\alpha-1+\frac\beta2}(\R^n)}.
\end{split}
\end{equation*}
\end{proof}
Finally, we apply Proposition \ref{pro1} with $\alpha=\frac12+$
and $\alpha=\frac12-$ to prove Theorem \ref{thm1}.\vspace{0.2cm}
\subsection{Proof of Theorem \ref{thm2}.} In this subsection, we
construct an example to show Theorem \ref{thm2}. The main idea is
the stationary phase argument. By \eqref{3.8}, we recall
\begin{equation}\label{3.17}
\begin{split} &u(x,t)=\sum_{k=0}^{\infty}\sum_{\ell=1}^{d(k)}Y_{k,\ell}(\theta)\int_0^\infty(r\rho)^{-\frac{n-2}2}J_{\nu(k)}(r\rho)e^{
it\rho^2}b^0_{k,\ell}(\rho)\rho^{n-1}\mathrm{d}\rho,
\end{split}
\end{equation}
where
\begin{equation*}b^0_{k,\ell}(\rho)=(\mathcal{H}_{\nu}a^0_{k,\ell})(\rho),\quad
u_0(x)=
u_0(r\theta)=\sum_{k=0}^{\infty}\sum_{\ell=1}^{d(k)}a^0_{k,\ell}(r)Y_{k,\ell}(\theta).
\end{equation*}
In particular we choose $u_0(x)$ to be a radial function such that
$(\mathcal{H}_{\nu(0)}u_0)(\xi)=\chi_{N}(|\xi|)$ where $\chi_N$ is a
smooth positive function supported in $J_N$ (to be chosen later) and
$N\gg1$. Then
\begin{equation}\label{3.18}
\begin{split} &u(x,t)=\int_0^\infty(r\rho)^{-\frac{n-2}2}J_{\nu(0)}(r\rho)e^{
it\rho^2}\chi_{N}(\rho)\rho^{n-1}\mathrm{d}\rho.
\end{split}
\end{equation}
Recalling the asymptotic expansion about the Bessel function
\begin{equation*}
J_\nu(r)=r^{-1/2}\sqrt{\frac2{\pi}}\cos(r-\frac{\nu\pi}2-\frac{\pi}4)+O_{\nu}(r^{-3/2}),\quad
\text{as}~ r\rightarrow\infty
\end{equation*}
with a constant depending on $\nu$ (see \cite{SW}), then we can
write
\begin{equation}\label{3.19}
\begin{split} u(x,t)&={C_{\nu}}\int_0^\infty(r\rho)^{-\frac{n-1}2}\big(e^{i(r\rho-\frac{\nu\pi}2-\frac{\pi}4)}-e^{-i(r\rho-\frac{\nu\pi}2-\frac{\pi}4)}\big)e^{
it\rho^2}\chi_{N}(\rho)\rho^{n-1}\mathrm{d}\rho\\&+C_{\nu}\int_0^\infty(r\rho)^{-\frac{n-2}2}O_{\nu}\big((r\rho)^{-\frac32}\big)e^{
it\rho^2}\chi_{N}(\rho)\rho^{n-1}\mathrm{d}\rho.
\end{split}
\end{equation}
Let us define
\begin{equation}\label{3.20}
I_1(r)=C_{\nu}e^{i(\frac{\nu\pi}2+\frac{\pi}4)}\int_0^\infty(r\rho)^{-\frac{n-1}2}e^{
i(-r\rho+t\rho^2)}\chi_{N}(\rho)\rho^{n-1}\mathrm{d}\rho,
\end{equation}
\begin{equation}\label{3.21}
I_2(r)=C_{\nu}e^{-i(\frac{\nu\pi}2+\frac{\pi}4)}\int_0^\infty(r\rho)^{-\frac{n-1}2}e^{
i(r\rho+t\rho^2)}\chi_{N}(\rho)\rho^{n-1}\mathrm{d}\rho,
\end{equation}
and
\begin{equation}\label{3.22}
I_3(r)=C_{\nu}\int_0^\infty(r\rho)^{-\frac{n-2}2}O_{\nu}\big((r\rho)^{-\frac32}\big)e^{
it\rho^2}\chi_{N}(\rho)\rho^{n-1}\mathrm{d}\rho.
\end{equation}
Let $\phi_r(\rho)=t\rho^2-r\rho$. The fundamental idea  is to choose
sets $J_N$ and $E\subset B^n$, in which $t(r)$ can be chosen, so
that $\partial_{\rho}\phi_r(\rho)= 2t(r)\rho-r $ almost vanishes for
all $\rho\in J_N$ and $r\in \{|x|:x\in E\}$. To this end, we choose
$$E=\{x: \frac1{100}\leq |x|\leq\frac18\}~~\text{and}~~
J_N=[N,N+2N^{\frac12}].$$ Choose $t(r)=\frac{r}{2(N+\sqrt{N})}$,
then $\partial_{\rho}\phi_r(N+N^{\frac12})= 0$. Thus
\begin{equation}\label{3.23}
I_1(r)=C_{\nu}e^{i(\frac{\nu\pi}2+\frac{\pi}4)}e^{
i\phi_r(N+\sqrt{N})}\int_0^\infty(r\rho)^{-\frac{n-1}2}e^{\frac{
ir[\rho-(N+\sqrt{N})]^2}{2(N+\sqrt{N})}}\chi_{N}(\rho)\rho^{n-1}\mathrm{d}\rho.
\end{equation}
Observe that
\begin{equation}\label{3.24}
|I_1(r)|\geq{c_{\nu}}\int_0^\infty(r\rho)^{-\frac{n-1}2}\cos{\big(\frac{
r[\rho-(N+\sqrt{N})]^2}{2(N+\sqrt{N})}\big)}\chi_{N}(\rho)\rho^{n-1}\mathrm{d}\rho.
\end{equation}
Moreover, there exists a small constant $c>0$ such that
\begin{equation*}
\cos{\big(\frac{ r[\rho-(N+\sqrt{N})]^2}{2(N+\sqrt{N})}\big)}\geq c,
\end{equation*}
since $|\frac{ r[\rho-(N+\sqrt{N})]^2}{2(N+\sqrt{N})}|\leq \frac
{\pi}4$ for all $\rho\in J_N$ with $N\gg1$ and $r\in
[\frac1{100},\frac18]$. Therefore,
\begin{equation}\label{3.25}
|I_1(r)|\geq c_\nu
r^{-\frac{n-1}2}\int_0^\infty\chi_{N}(\rho)\rho^{\frac{n-1}2}\mathrm{d}\rho
\geq c_\nu r^{-\frac{n-1}2}N^{\frac {n}2}.
\end{equation}
On the other hand, let $\varphi_{r}(\rho)=t\rho^2+r\rho$, $t=t(r)$
as before, then
$\partial_\rho\varphi_{r}(\rho)=2t(r)\rho+r\geq\frac1{200}$ when
$\rho\in J_N$ and $r\in[\frac1{100},\frac18]$. From the integral by
parts, we obtain
\begin{equation}\label{3.26}
|I_2(r)| \leq C_\nu r^{-\frac{n}2}N^{\frac {n-2}2}.
\end{equation}
Obviously, we have
\begin{equation}\label{3.27}
|I_3(r)| \leq C_\nu r^{-\frac{n}2}N^{\frac {n-2}2}.
\end{equation}
Combining \eqref{3.25}-\eqref{3.27}, we get for $N\gg1$ and
$r\in[\frac1{100},\frac18]$
\begin{equation}\label{3.28}
u^*(x)\geq c N^{\frac {n}2}.
\end{equation}
On the other hand, let $j_0=\log_2 N$, then we obtain by the definition
of $P_j$ and $\tilde{P}_j$
\begin{equation*}
\|u_0(x)\|^2_{H^s} =\sum_{j}2^{2js}\|P_j u_0\|^2_{
L^{2}}=\sum_{j}2^{2js}\|P_j \tilde{P}_{j_0}u_0\|^2_{ L^{2}}.
\end{equation*}
By Lemma \ref{orthogonality}, we choose
$s<\epsilon<1+\min\{\frac{n-2}2, (\frac{(n-2)^2}4+a)^{\frac12}\}$ to
obtain
\begin{equation}\label{3.29}
\begin{split}
\|u_0(x)\|^2_{H^s}& \leq
C\sum_{j}2^{2js-2\epsilon|j-j_0|}\|u_0\|^2_{
L^{2}}\\&=CN^{2s}\sum_{j}2^{2js-2\epsilon|j|}\|\chi_N\|^2_{
L^{2}}=N^{2s+n-\frac12}.
\end{split}
\end{equation}
Thus, by \eqref{1.5} and \eqref{3.28}, we must have $s\geq1/4$.
\subsection{Proof of the Theorem \ref{thm3}.} In this subsection, we
show Theorem \ref{thm3}. Even though there is a loss of the angular
regularity in Theorem \ref{thm3}, the result implies the sharp
result for the radial initial data. The key ingredient here is the
following lemma proved in \cite{GS}:
\begin{lemma}\label{lem1}Let $\tilde{J}_{\nu}(s)=s^{\frac12}J_{\nu}(s)$ with $s\geq0$, and let
\begin{equation}\label{4.9}
\begin{split}
T_{\nu}g(r)=\int_{I}\frac{e^{it(r)\rho^2}\tilde{J}_{\nu}(r\rho)}{\rho^{\frac14}}g(\rho)\mathrm{d}\rho.
\end{split}
\end{equation}
Then
\begin{equation}\label{4.9}
\begin{split} \int_{0}^1\big|T_{\nu}g(r)\big|^2\mathrm{d}r\leq
C\int_{I}|g(\rho)|^2\mathrm{d}\rho,
\end{split}
\end{equation}
where the constant $C$ is independent of $g\in L^2(I)$, of the
interval $I$, of the measurable function $t(r)$ and of the order
$\nu\geq0$.
\end{lemma}
We also can follow the Carleson approach \cite{Car} to linearize our
maximal operator, by making t into a function of $r$, $t(r)$. By the
triangle inequality, we estimate
\begin{equation*}
\begin{split} \|u^*(x)\|_{L^2(B^n)}\leq C\sum_{k=0}^{\infty}\sum_{\ell=1}^{d(k)}\Big\|\int_0^\infty(r\rho)^{-\frac{n-2}2}J_{\nu(k)}(r\rho)e^{
it(r)\rho^2}b^0_{k,\ell}(\rho)\rho^{n-1}\mathrm{d}\rho\Big\|_{L^2_{r^{n-1}\mathrm{d}r}}.
\end{split}
\end{equation*}
Let $g(\rho)=b^0_{k,\ell}(\rho)\rho^{\frac{n-1}2+\frac14}$, then
\begin{equation}\label{4.9}
\begin{split} \|u^*(x)\|_{L^2(B^n)}\lesssim\sum_{k=0}^{\infty}\sum_{\ell=1}^{d(k)}\Big\|\int_0^\infty \tilde{J}_{\nu(k)}(r\rho)e^{
it(r)\rho^2}\rho^{-\frac14}g(\rho)\mathrm{d}\rho\Big\|_{L^2_{r}([0,1])}.
\end{split}
\end{equation}
Using Lemma \ref{lem1}, we obtain
\begin{equation}\label{4.9}
\begin{split} \|u^*(x)\|_{L^2(B^n)}\lesssim
C\sum_{k=0}^{\infty}\sum_{\ell=1}^{d(k)}\Big\|b^0_{k,\ell}(\rho)\rho^{\frac{n-1}2+\frac14}\Big\|_{L^2_{\rho}(\R^+)}.
\end{split}
\end{equation}
Let $\alpha=(n-1)/2+\epsilon$ with $\epsilon>0$, we have by the
Cauchy-Schwarz inequality
\begin{equation*}
\begin{split} \|u^*(x)\|_{L^2(B^n)}\leq
C\Big(\sum_{k=0}^{\infty}\sum_{\ell=1}^{d(k)}(1+k)^{-2\alpha}\Big)^{\frac12}\Big(\sum_{k=0}^{\infty}\sum_{\ell=1}^{d(k)}
(1+k)^{2\alpha}\Big\|b^0_{k,\ell}(\rho)\rho^{\frac{n-1}2+\frac14}\Big\|^2_{L^2_{\rho}(\R^+)}\Big)^{\frac12}.
\end{split}
\end{equation*}
Since $d(k)\simeq \langle k\rangle^{n-2}$, we have by Lemma
\ref{sobolev}
\begin{equation*}
\begin{split} \|u^*(x)\|_{L^2(B^n)}&\lesssim
\Big(\sum_{k=0}^{\infty}\sum_{\ell=1}^{d(k)}(1+k)^{2\alpha}\Big\|b^0_{k,\ell}(\rho)\rho^{\frac{n-1}2+\frac14}\Big\|^2_{L^2_{\rho}(\R^+)}\Big)^{\frac12}
\lesssim \|u_0\|_{H^{\frac14}_rH^{\alpha}_\theta}.
\end{split}
\end{equation*}
This completes the proof of Theorem \ref{thm3}.

\begin{center}

\end{center}
\end{document}